\documentclass[12pt,twoside,a4paper]{amsart}

\usepackage[margin=2.5cm]{geometry}

\usepackage{amssymb,mathtools}

 \usepackage[dvipsnames]{xcolor} 
\usepackage[colorlinks]{hyperref}
\usepackage[color=Orange!50!white, textsize=scriptsize]{todonotes}

\newtheorem{thm}{Theorem}[section]

\newtheorem{prop}[thm]{Proposition}
\theoremstyle{definition}

\newtheorem{conjecture}[thm]{Conjecture}

\newtheorem{question}[thm]{Question}
\newtheorem{remark}[thm]{Remark}

\numberwithin{equation}{section}
\allowdisplaybreaks

\begin{document}


\title[Nonempty-interior problems for product sets]{Peres--Schlag's nonempty-interior problem and a shifted-product variant for product sets}
\author{Guo-Dong Hong} 
\address{Department of Mathematics, California Institute of Technology, Pasadena, CA 91125,
USA}
\email{ghong@caltech.edu}

\author{Chong-Wei Liang}
\address{Department of Mathematics, National Taiwan University, Taiwan.}
\email{d10221001@ntu.edu.tw}

\author{Chun-Yen Shen} 
\address{Department of Mathematics, National Taiwan University, Taiwan.}
\email{cyshen@math.ntu.edu.tw}
\date{}

\begin{abstract}
We study finite-field analogues of the Peres--Schlag nonempty-interior problem
for product sets. Given \(A\subseteq\mathbb F_p\), we ask when a suitable
one-dimensional linear image of \(A^n\) is full; equivalently, when there exist
coefficients \(t_1,\ldots,t_n\in\mathbb F_p\) such that
\[
    t_1A+\cdots+t_nA=\mathbb F_p.
\]
For \(n\ge3\), we prove that this holds whenever
\[
    |A|\gg_n p^{\frac{3}{2n-1}}.
\]
This improves the exponent predicted by the direct product-set analogue of the
Peres--Schlag threshold, namely \(|A|\gg p^{2/n}\). We also prove a
two-dimensional near-half-density result.

Motivated by sum-product phenomena, we also introduce and study a product-type
variant in which linear forms are replaced by shifted product maps. We prove
finite-field covering results for shifted products
\[
    (t_1 + A)(t_2 + A)\cdots(t_n + A)
\]
at the same density scale as in the linear case. Finally, we prove a Euclidean
shifted-product analogue: if \(A\subseteq\mathbb R\) is Borel and
\(\dim_H A>2/n\), then some shifted product of \(n\) copies of \(A\) contains a
nonempty open interval.
\end{abstract}

\maketitle

\section{Introduction and the main results}\label{sec:0}

The study of projections occupies a central place in geometric measure theory and fractal geometry \cite{MR3558147,MR0248779,MR0409774,MR0063439}. A fundamental theme is understanding how the Hausdorff dimension of a set is reflected in its projections onto lower-dimensional subspaces.
Let $E\subseteq\mathbb R^n$ be a Borel set, and let $\pi_V:\mathbb R^n\to V$ denote orthogonal projection onto a $k$-dimensional subspace $V\in G(n,k)$. The classical Marstrand--Mattila projection theorem asserts that, for almost every $V\in G(n,k)$,
$$
\dim_H \pi_V(E)=\min\{\dim_H E,k\}.
$$
Moreover, if $\dim_H E>k$, then $\pi_V(E)$ has positive $k$-dimensional Lebesgue measure for almost every $V\in G(n,k)$.

A stronger question asks when the projections have nonempty interior in the target subspace. Peres and Schlag \cite{MR1749437} proved that if
$$
\dim_H E>2k,
$$
then $\pi_V(E)$ has nonempty interior for almost every $V\in G(n,k)$. See also the survey \cite{MR2044636} for background on nonempty-interior problems for projections. For projections onto lines, this gives the threshold $\dim_H E>2$, which is sharp in general.

Finite-field analogues of these projection questions were studied by
Chen~\cite{MR3753167}.  In the finite-field setting, the analogue of having nonempty
interior is being {\it full}: a projection is full if its image is the whole target
space.  Chen proved, in particular, a finite-field analogue of the
Peres--Schlag threshold.  If \(E\subseteq\mathbb F_p^n\) has size
\[
    |E| \gg p^s
    \quad\text{with}\quad
    s>2k,
\]
then almost every \(k\)-dimensional finite-field projection of \(E\) is full. A recent result of the authors \cite{HLS} establishes a sharp finite-field projection theorem. If \(E\subseteq\mathbb F_p^n\) has size
\[
    |E|\gg p^s
    \quad\text{with}\quad
    s>k+1,
\]
then almost every \(k\)-dimensional finite-field projection of \(E\) is full. Moreover, the threshold $k+1$ is sharp, as shown by Kakeya-type constructions over finite fields.
\smallskip

The purpose of this paper is to study this finite-field projection problem for
sets with product structure.  Let \(A\subseteq\mathbb F_p\), and consider
\(E=A^n\subseteq\mathbb F_p^n\).  A one-dimensional linear image of \(A^n\)
has the form
$ (x_1,\ldots,x_n)
    \mapsto
    t_1x_1+\cdots+t_nx_n$,
and its image is the dilated sumset
$ t_1A+\cdots+t_nA$.
Thus, in the product-set setting, the full-projection problem becomes the
following additive covering question: when can one choose coefficients
\(t_1,\ldots,t_n\in\mathbb F_p\) such that
\[
    t_1A+\cdots+t_nA=\mathbb F_p?
\]

For a general set \(E\subseteq\mathbb F_p^n\), the Peres--Schlag--Chen
threshold for one-dimensional projections corresponds to the scale
$|E|\gg p^2$.
For \(E=A^n\), this suggests the product-set threshold
$$|A|\gg p^{2/n}.$$
One of the main points of this paper is that product structure allows one to go
below this scale.
For every \(n\ge3\), we prove full coverage under the
hypothesis
\[
    |A|\gg_{n} p^{\frac{3}{2n-1}}.
\]
Note that
\[
    \frac{3}{2n-1}<\frac{2}{n}
\]
for all $n\geq 3$, and this gives an exponent improvement over the direct
product-set analogue of the Peres--Schlag--Chen threshold.

Motivated by sum-product phenomena \cite{MR2053599,MR1948103}, we also study
a product-type variant of the same question.  Instead of linear forms, we
consider shifted product maps
$(x_1,\ldots,x_n)
    \mapsto
    (t_1+x_1)(t_2+x_2)\cdots(t_n+x_n)$.
The image of \(A^n\) under such a map is the shifted product set
\[
    (t_1 + A)(t_2 + A)\cdots(t_n + A).
\]
We prove finite-field covering results for these product-type projections
parallel to the linear results.  We also prove a Euclidean product-type result
for Borel subsets of \(\mathbb R\).

\subsection{Linear projections}

We first state the finite-field results for linear projections of product sets.
The two-dimensional case has a natural half-density threshold.  Indeed, if
\(|A|\ge (p+1)/2\), then the Cauchy--Davenport inequality implies 
\[
    A+tA=\mathbb F_p
\]
for every \(t\in\mathbb F_p^\times\).  Our first result shows that, after
allowing the dilation \(t\) to depend on \(A\), this conclusion remains true
slightly below the half-density threshold.

\begin{thm}
\label{thm:half-density-dilated-sumset-covering}
There is an absolute constant \(\varepsilon_0>0\) such that, for all sufficiently
large primes \(p\), the following holds.  If \(A\subseteq \mathbb F_p\)
satisfies
\[
    |A|>\left(\frac12-\varepsilon_0\right)p,
\]
then there exists \(t\in\mathbb F_p^\times\) such that $A+tA=\mathbb F_p$.
For example, one may take \(\varepsilon_0=10^{-3}\), provided \(p\) is
sufficiently large.
\end{thm}

This near-density statement improves the elementary half-density bound
by a fixed constant. It is natural to ask whether one can obtain an
exponent saving in the two-dimensional problem.

\begin{question}\label{ques:two-dimensional-exponent-saving}
Is there an absolute constant $\delta>0$ such that the following
holds for all sufficiently large primes $p$? If $A\subseteq\mathbb F_p$
satisfies
\[
    |A|\gg p^{1-\delta},
\]
then there exists $t\in\mathbb F_p^\times$ such that $A+tA=\mathbb F_p$.
\end{question}

\begin{remark}\label{rem:shkredov-example}
After the first version of this paper appeared on arXiv, Ilya Shkredov
communicated to us the following example. We are grateful
to him for sharing this example with us and for allowing us to present
it here.

Let $p\equiv 3\pmod 4$, and let $R$ and $N$ denote the sets of
nonzero quadratic residues and nonresidues in $\mathbb F_p$,
respectively. Set $A=R\cap(1-R)$. We first show that $|A|=(p-3)/4$.
Since $-1\in N$, the elements of $R\setminus\{1\}$ occur in pairs
$\{r,r^{-1}\}$ with $r\ne r^{-1}$. For each such pair, the identity
$1-r^{-1}=-(1-r)/r$ shows that exactly one of $1-r$ and $1-r^{-1}$
lies in $R$. Hence exactly one of $r$ and $r^{-1}$ belongs to $A$,
and $|A|=(|R|-1)/2=(p-3)/4$.

We now show that $A+tA\ne\mathbb F_p$ for every
$t\in\mathbb F_p^\times$. We consider two cases.
\begin{itemize}
    \item If $t\in R$, then $A+tA\subseteq R+R$. Since $-1\in N$,
    we have $0\notin R+R$, so $0\notin A+tA$.
    \item If $t\in N$, then $tA\subseteq N$ and $1-A\subseteq R$.
    Hence $tA\cap(1-A)=\varnothing$, so $1\notin A+tA$.
\end{itemize}
Since $|A|=(p-3)/4$ and there are arbitrarily large primes
$p\equiv 3\pmod 4$, this gives a negative answer to
Question~\ref{ques:two-dimensional-exponent-saving}.
\end{remark}

In higher dimensions, the additional summands allow a genuine improvement over
the scale \(p^{2/n}\) suggested by the general projection threshold.  This is
the content of the next theorem.

\begin{thm}
\label{thm:dilated-n-fold-covering-improved}
Let \(n\ge 3\). Then there exist constants
\(C=C(n)>0\) and \(p_0=p_0(n)\) such that the following holds for
every prime \(p\ge p_0\).

If \(A\subseteq \mathbb F_p\) satisfies
\[
    |A|\ge C p^{\frac{3}{2n-1}},
\]
then there exist \(t_2,\ldots,t_n\in\mathbb F_p^\times\) such that
\[
    A+t_2A+\cdots+t_nA=\mathbb F_p.
\]
\end{thm}

Theorem~\ref{thm:dilated-n-fold-covering-improved} suggests that the product
structure of \(A^n\) should also be relevant in the Euclidean setting.  This
motivates the following conjecture.

\begin{conjecture}
\label{conj:linear-projection-interior}
For every \(n\ge 3\), there exists \(\varepsilon_n>0\) such that the following
holds.  If \(A\subseteq \mathbb R\) is a Borel set satisfying
\[
    \dim_H A>\frac{2}{n}-\varepsilon_n,
\]
then there exist \(t_2,\ldots,t_n\in\mathbb R\) such that
$A+t_2A+\cdots+t_nA$
contains a nonempty open interval.
\end{conjecture}

\subsection{Product-type projections}

We next consider the shifted product analogue.  For \(A\subseteq\mathbb F_p\),
the problem is to choose shifts \(t_1,\ldots,t_n\in\mathbb F_p\) so that
\[
    (t_1 + A)(t_2 + A)\cdots(t_n + A)=\mathbb F_p.
\]
The first result is the two-fold analogue of
Theorem~\ref{thm:half-density-dilated-sumset-covering}.

\begin{thm}
\label{thm:half-density-shifted-product-covering}
There is an absolute constant \(\varepsilon_0>0\) such that, for all sufficiently
large primes \(p\), the following holds.  If \(A\subseteq \mathbb F_p\)
satisfies
\[
    |A|>\left(\frac12-\varepsilon_0\right)p,
\]
then there exist \(s,t\in\mathbb F_p\) such that
$(A+s)(A+t)=\mathbb F_p$.
For example, one may take \(\varepsilon_0=10^{-3}\), provided \(p\) is
sufficiently large.
\end{thm}

As in the linear case, this gives a small improvement below the trivial
half-density range.  The corresponding \(n\)-fold shifted product result holds
at the same density scale as Theorem~\ref{thm:dilated-n-fold-covering-improved}.

\begin{thm}
\label{thm:shifted-product-covering}
Let \(n\ge 3\). Then there exist constants
\(C=C(n)>0\) and \(p_0=p_0(n)\) such that the following holds for
every prime \(p\ge p_0\).

If \(A\subseteq\mathbb F_p\) satisfies
\[
    |A|\ge C p^{\frac{3}{2n-1}},
\]
then there exist \(t_1,\ldots,t_n\in\mathbb F_p\) such that
\[
    (t_1+A)(t_2+A)\cdots(t_n+A)=\mathbb F_p.
\]
\end{thm}

There is also a continuous counterpart of the product-type problem.  This leads to the following Euclidean product-type theorem.

\begin{thm}
\label{thm:continuous-shifted-product-interior}
Let \(n\ge 3\), and let \(A\subseteq \mathbb R\) be a Borel set satisfying
\[
    \dim_H A>\frac{2}{n}.
\]
Then there exist \(t_1,\ldots,t_n\in\mathbb R\) such that
$(t_1 + A)(t_2 + A)\cdots(t_n + A)$
contains a nonempty open interval.
\end{thm}

Theorem~\ref{thm:continuous-shifted-product-interior} reaches the direct
threshold \(\dim_H A>2/n\).  In analogy with the finite-field results above,
one expects this threshold not to be optimal for product-type projections.  We
record this expected improvement as a conjecture.

\begin{conjecture}
\label{conj:shifted-product-interior}
For every \(n\ge 3\), there exists \(\varepsilon_n>0\) such that the following
holds.  If \(A\subseteq \mathbb R\) is a Borel set satisfying
\[
    \dim_H A>\frac{2}{n}-\varepsilon_n,
\]
then there exist \(t_1,\ldots,t_n\in\mathbb R\) such that
$ (t_1 + A)(t_2 + A)\cdots(t_n + A)$
contains a nonempty open interval.
\end{conjecture}

\noindent\textbf{Outline of ideas.}
Here we only outline the proofs of the two linear projection results,
Theorems~\ref{thm:half-density-dilated-sumset-covering}
and~\ref{thm:dilated-n-fold-covering-improved}. The proofs of the
corresponding shifted-product results,
Theorems~\ref{thm:half-density-shifted-product-covering}
and~\ref{thm:shifted-product-covering}, follow similar arguments.
For the Euclidean result, Theorem~\ref{thm:continuous-shifted-product-interior},
we use a classical Fourier-analytic argument based on Sobolev embedding.

We begin with the two-dimensional case. The key idea in
Theorem~\ref{thm:half-density-dilated-sumset-covering} is a rigidity
argument. If $A+tA\ne\mathbb F_p$ for every $t\in\mathbb F_p^\times$,
the missing points give rise to an affine map of each nonzero slope
that almost preserves $A$. From these maps we deduce that $A$ is
approximately invariant either under all dilations with the same
center or under all translations. In the first case, we translate
the coordinates so that this center is the origin. A direct count
gives
\[
    \frac{1}{p-1}\sum_{t\ne0}|A\cap tA|
    = \mathbf 1_A(0)+\frac{|A\cap\mathbb F_p^\times|^2}{p-1}.
\]
When $|A|$ is close to $p/2$, this identity shows that the average
size of $A\cap tA$ is roughly $p/4$. On the other hand, since $tA$
almost coincides with $A$ for every $t\ne0$, the same average must
be close to $|A|$, and hence to $p/2$. This gives a contradiction
in the dilation case. In the translation case, averaging over all
translates of $A$ gives the same contradiction.

We now turn to the higher-dimensional case. To prove
Theorem~\ref{thm:dilated-n-fold-covering-improved}, we first average
the additive energy over the dilations $t_2,\ldots,t_{n-1}$.
This gives a choice of $\mathbf t$ for which the sumset
$S_{\mathbf t}=A+t_2A+\cdots+t_{n-1}A$ has complement
$H=\mathbb F_p\setminus S_{\mathbf t}$ satisfying
\[
    |H|\le 1+\frac{p^2}{|A|^{n-1}}.
\]
If $|H|<|A|$, then $x-sA$ meets $S_{\mathbf t}$ for every
$x\in\mathbb F_p$ and every $s\ne0$. Thus, $S_{\mathbf t}+sA=\mathbb F_p$
for every $s\ne0$, and comparing the displayed bound with $|A|$
recovers the exponent $2/n$. The improvement comes from showing that
if no choice of $s\ne0$ gives full coverage, then $|H|$ must be at
least of order $(p|A|)^{1/2}$. Indeed, in that case we may choose
$x_s\notin S_{\mathbf t}+sA$ for each $s\ne0$. Then $x_s-sA\subseteq H$,
so the missing points give rise to $p-1$ lines of distinct slopes,
each containing $|A|$ points of $A\times H$. Applying the incidence
estimate of Stevens and de Zeeuw gives the claimed lower bound for
$|H|$. Comparing it with the displayed upper bound yields the improved
exponent $3/(2n-1)$.

\subsection*{Organization}

In Section~\ref{section:Preliminaries}, we collect the notation and the background material from the
incidence geometry and additive combinatorics.  In
Section~\ref{section:Linear projections}, we prove the two linear projection
results, Theorems~\ref{thm:half-density-dilated-sumset-covering} and
\ref{thm:dilated-n-fold-covering-improved}.  In
Section~\ref{section:Product-type projections}, we prove the finite-field
shifted product results, Theorems~\ref{thm:half-density-shifted-product-covering}
and \ref{thm:shifted-product-covering}.  Finally, in
Section~\ref{section:Euclidean product-type projections}, we prove
Theorem~\ref{thm:continuous-shifted-product-interior}.


\section{Preliminaries} \label{section:Preliminaries}

\subsection{Notation and asymptotic conventions}
Throughout the paper, \(p\) denotes a prime number, and all subsets are subsets
of \(\mathbb F_p\) unless otherwise specified.  We write
\(\mathbb F_p^\times:=\mathbb F_p\setminus\{0\}\).  For a finite set \(X\), we
write \(|X|\) for its cardinality.  For two sets \(A,B\), we write
\[
    A\triangle B:=(A\setminus B)\cup(B\setminus A)
\]
for their symmetric difference.

All asymptotic notation is taken as \(p\to\infty\) through primes.  The
parameters appearing in the statement of a theorem, such as \(n\),
and \(\varepsilon_0\), are regarded as fixed, while the set \(A\) may depend on
\(p\).

We use Vinogradov notation for upper bounds: \(F\ll G\) means that
$|F|\le C G$
for all sufficiently large \(p\), with an absolute constant \(C>0\).  Subscripts
indicate the permitted dependence of the implicit constant; for instance,
\(F\ll_n G\) means that the implicit constant may depend on \(n\), but not on
\(p\) or on \(A\).  We write \(F\gg G\) to mean \(G\ll F\).

If \(G=G(p)>0\), then \(F=o(G)\) means
\[
    \frac{F}{G}\to0,
    \quad\text{as }p\to\infty.
\]

\subsection{Cauchy--Davenport inequality}
We record the Cauchy--Davenport inequality (see, e.g., \cite[Chapter~5]{MR2289012}).

\begin{prop}\label{prop:CD-ineq}
If $p$ is a prime and $A,B\subseteq\mathbb F_p$ are nonempty subsets, then
\[
    |A+B|\geq\min\{|A|+|B|-1,\,p\}.
\]
Moreover, if $A',B'\subseteq\mathbb F_p^\times$ satisfy
$|A'|+|B'|>p-1$, then
\[
    A'B'=\mathbb F_p^\times.
\]
\end{prop}

\subsection{Stevens--de Zeeuw Cartesian-product incidence estimate}
We will use the following incidence estimate for Cartesian-product sets over finite fields due to Stevens--de Zeeuw \cite{MR3742451}.
\begin{prop}
\label{prop:sdz-cartesian-product-incidence}
Let \(X,Y\subseteq\mathbb F_p\) with \(|X|\le |Y|\), and let \(\mathcal L\)
be a finite set of lines in \(\mathbb F_p^2\).  Suppose that
    $|X||Y|^2\le |\mathcal L|^3$ and
    $|X||\mathcal L|\le p^2$, then
\[
    \mathcal I(X\times Y,\mathcal L)
    \ll
    |X|^{3/4}|Y|^{1/2}|\mathcal L|^{3/4}
    +
    |\mathcal L|,
\]
where the incidence is given by
\[
    \mathcal I(X\times Y,\mathcal L)
    :=
   \left |\{((x,y),\ell)\in (X\times Y)\times\mathcal L:(x,y)\in\ell\}\right|.
\]
\end{prop}
\smallskip


\section{Linear projections over finite fields}
\label{section:Linear projections}

In this section, we prove the two finite-field linear projection results stated
in the introduction.  Both are covering statements for dilated sumsets, but the
two arguments use different mechanisms.

The first result treats the two-fold sumset \(A+tA\) in the near-half-density
regime.  Its proof is based on a {\it rigidity argument}: if every set \(A+tA\) misses
a point, then the missing points give rise to a family of affine maps that
almost preserve \(A\).  This approximate invariance is then ruled out by
averaging over translations and dilations.

The second result treats \(n\)-fold sums with \(n\ge3\).  Here, the additional
summands allow one to work at a much lower density.  We first construct an
\((n-1)\)-fold dilated sumset with small complement, and then use a
Cartesian product incidence estimate to fill the remaining elements.

We begin with the near-half-density result.

\begin{proof}[Proof of Theorem~\ref{thm:half-density-dilated-sumset-covering}]
If \(|A|\ge (p+1)/2\), then the conclusion follows immediately from Proposition ~\ref{prop:CD-ineq}.  We may
therefore assume that 
\[
\left(\frac12-\varepsilon_0\right)p<|A|<(p+1)/2.
\]
Suppose, for contradiction, that \(A+tA\ne\mathbb F_p\) for every
\(t\in\mathbb F_p^\times\).  For each \(t\in\mathbb F_p^\times\), choose
\(x_t\in\mathbb F_p\setminus(A+tA)\) and let $B:=\mathbb F_p\setminus A$.  Then
\[
    x_t-tA\subseteq B,
    \quad
    \text{for every }t\in\mathbb F_p^\times.
\]
In particular, \(x_1-A\) and \(x_t-tA\) are both subsets of \(B\), each of cardinality \(|A|\).  Hence,
\[
    |(x_1-A)\cap(x_t-tA)|
    \ge
    2|A|-|B|
    =
    3|A|-p;
\]
and, applying the affine bijection \(y\mapsto x_1-y\), we get
\begin{align}\label{estimate1}
    |A\cap(tA+d_t)|\ge 3|A|-p,
    \quad\text{where}\,\,
    d_t:=x_1-x_t.
\end{align}
Define \(g_t:\mathbb F_p\to\mathbb F_p\) by \(g_t(x):=tx+d_t\).  Then
\(g_t(A)=tA+d_t\), and therefore (\ref{estimate1}) gives
\[
    |A\triangle g_t(A)|
    \le
    2|A|-2(3|A|-p)
    =p\delta,
    \quad
    \text{for all}\,t\in\mathbb F_p^\times,
\]
where \(\delta=(2p-4|A|)/p\).  The size assumption on $A$ gives
\begin{equation}
    \delta<4\varepsilon_0.
    \label{eq:half-density-dilated-delta-small}
\end{equation}
For \(s,t\in\mathbb F_p^\times\), define
$h_{s,t}:=g_{st}^{-1}g_sg_t$.
This map is a translation.  By the triangle inequality,
\begin{align}\label{eq:half-density-dilated-defect-bound}
    |A\triangle h_{s,t}(A)|&=
    |g_{st}(A)\triangle g_sg_t(A)|\notag\\
    &\le
    |g_{st}(A)\triangle A|
    +|A\triangle g_s(A)|
    +|g_s(A)\triangle g_sg_t(A)|\notag\\&\le 3\delta p,
\end{align}
where the first equality follows because \(g_{st}\) is a bijection.  We now derive a contradiction by considering two cases. 

Suppose that \(h_{s,t}(x)=x+r\) for some \(s,t\in\mathbb F_p^\times\) and
some \(r\in\mathbb F_p^\times\).  Write \(\tau_r(x):=x+r\). 
For \(u\in\mathbb F_p^\times\), conjugating \(\tau_r\) by \(g_u\) gives
$g_u\tau_rg_u^{-1}(x)=x+ur$.
Thus, these conjugates range over all nonzero translations as \(u\) ranges over
\(\mathbb F_p^\times\).
Using the triangle inequality and
(\ref{eq:half-density-dilated-defect-bound}), we get
\[
\begin{aligned}
    |A\triangle g_u\tau_rg_u^{-1}(A)|
    &\le
    |A\triangle g_u(A)|
    +|g_u(A)\triangle g_u\tau_r(A)| 
    +|g_u\tau_r(A)\triangle g_u\tau_rg_u^{-1}(A)|  \\
    &\le
    5\delta p,
\end{aligned}
\]
which implies that
$|A\triangle(A+v)|\le 5\delta p$ for all 
    $v\in\mathbb F_p^\times$, and hence,

\begin{equation}
    \frac{2|A|(p-|A|)}{p-1}=\frac1{p-1}\sum_{v\in\mathbb F_p^\times}|A\triangle(A+v)|\leq 5\delta p.
    \label{eq:half-density-dilated-translation-average}
\end{equation}
Using the size assumption on $A$, \eqref{eq:half-density-dilated-translation-average}, and \eqref{eq:half-density-dilated-delta-small}, we obtain 
\[
   \left(\frac12-3\varepsilon_0\right)p
   \leq \frac{2|A|(p-|A|)}{p-1}<
   5\delta p< 20\varepsilon_0p,
\]
for all sufficiently large \(p\).
For \(\varepsilon_0=10^{-3}\), these estimates are incompatible for all
sufficiently large \(p\).  This is a contradiction. 

Thus, we may assume that all the maps $h_{s,t}$ are trivial, that is,
\[
    h_{s,t}=\mathrm{id}_{\mathbb F_p},
    \quad\text{for all}\,s,t\in\mathbb F_p^\times,
\]
or equivalently \(g_sg_t=g_{st}\) for all \(s,t\in\mathbb F_p^\times\).
The identity \(g_sg_t=g_{st}\) gives
\(d_{st}=sd_t+d_s\).  
Since the multiplicative group $\mathbb F_p^\times$ is cyclic \cite[Theorem~2.8]{MR1429394},
we may choose a generator $\lambda$ of $\mathbb F_p^\times$.
Set $c:=\frac{d_\lambda}{1-\lambda}$.
For every $t\in\mathbb F_p^\times$, there exists an integer
$j\in\{0,\ldots,p-2\}$, depending on $t$, such that $t=\lambda^j$.
Recall that $d_{st}=sd_t+d_s$. In particular, $d_1=0$, and
\[
    d_{\lambda^{j+1}}
    =\lambda^j d_\lambda+d_{\lambda^j}.
\]
By induction, we obtain
\[
    d_{\lambda^j}
    =(1+\lambda+\cdots+\lambda^{j-1})d_\lambda
    =\frac{1-\lambda^j}{1-\lambda}d_\lambda
    =(1-t)c.
\]
Hence, $g_t(x)=tx+(1-t)c=c+t(x-c)$, so all the maps $g_t$
fix the common point $c$. Since translation preserves cardinality, after translating the coordinates and
renaming \(A-c\) as \(A\), we may assume
that \(g_t(x)=tx\) and
\begin{equation}
   |A\triangle g_t(A)|= |A\triangle tA|\le \delta p,
    \quad
    \text{for every }t\in\mathbb F_p^\times.
    \label{eq:half-density-dilated-dilation-bound}
\end{equation}
Let \(k:=|A\cap\mathbb F_p^\times|\).  Since multiplication by
\(t\in\mathbb F_p^\times\) fixes \(0\) and acts transitively on
\(\mathbb F_p^\times\), we have
\[
    \sum_{t\ne0}|A\cap tA|
    =
    (p-1)\mathbf 1_A(0)+k^2.
\]
It follows from \(|A|=k+\mathbf 1_A(0)\) and \eqref{eq:half-density-dilated-dilation-bound} that 
\begin{equation}
\frac{2k(p-1-k)}{p-1}=
    \frac1{p-1}\sum_{t\ne0}|A\triangle tA|
    <\delta p.
    \label{eq:half-density-dilated-dilation-average}
\end{equation}
Furthermore, since \(k\in\{|A|,|A|-1\}\), the size assumption on $A$, together with \eqref{eq:half-density-dilated-dilation-average} and \eqref{eq:half-density-dilated-delta-small}, gives
\[
     \left(\frac12-5\varepsilon_0\right)p\leq \frac{2k(p-1-k)}{p-1}
    <
    \delta p
    <
    4\varepsilon_0 p,
\]
for all sufficiently large \(p\).  
For \(\varepsilon_0=10^{-3}\), these estimates are incompatible for all
sufficiently large \(p\), which is again a contradiction.  

As a consequence,
there exists \(t\in\mathbb F_p^\times\) such that
\(
    A+tA=\mathbb F_p.
\)
This completes the proof.
\end{proof}

We now turn to the higher-dimensional theorem.  Unlike the preceding
near-density result, this is a genuinely lower-density statement: the additional
summands are used to build a large intermediate sumset before the final covering step.

\begin{proof}[Proof of Theorem~\ref{thm:dilated-n-fold-covering-improved}]
We first choose an \((n-1)\)-fold dilated sumset, with the first coefficient
normalized to \(1\), whose complement is small, and then use the incidence
estimate, Proposition~\ref{prop:sdz-cartesian-product-incidence}, to show that one
more dilate fills this complement.

Set \(k:=n-1\geq 2\). For
\(\mathbf t=(t_2,\ldots,t_k)\in(\mathbb F_p^\times)^{k-1}\), define
\[
    S_{\mathbf t}:=A+t_2A+\cdots+t_kA.
\]
For \(x\in\mathbb F_p\), let
\[
    r_{\mathbf t}(x)
    :=
    \#\{(a_1,\ldots,a_k)\in A^k:
    a_1+t_2a_2+\cdots+t_ka_k=x\},
\]
and set the additive energy
\(E_{\mathbf t}:=\sum_x r_{\mathbf t}(x)^2\).
By the Cauchy--Schwarz inequality,
\begin{align}\label{CS}
    |S_{\mathbf t}|
    \ge
    \frac{|A|^{2k}}{E_{\mathbf t}}.
\end{align}
We average \(E_{\mathbf t}\) over
\((\mathbb F_p^\times)^{k-1}\). The quantity \(E_{\mathbf t}\) counts tuples
$(a_1,\ldots,a_k,b_1,\ldots,b_k)\in A^{2k}$
satisfying
\[
    (a_1-b_1)+t_2(a_2-b_2)+\cdots+t_k(a_k-b_k)=0.
\]
The diagonal tuples contribute \((p-1)^{k-1}|A|^k\). For a non-diagonal tuple,
not all of \(a_2-b_2,\ldots,a_k-b_k\) can vanish; otherwise the equation would
force \(a_1=b_1\). Hence, after choosing an index \(j\ge2\) with
\(a_j\ne b_j\), all shifts except \(t_j\) determine \(t_j\) uniquely. Thus, each
non-diagonal tuple contributes at most \((p-1)^{k-2}\) choices of
\(\mathbf t\). Therefore,
\[
    \frac1{(p-1)^{k-1}}
    \sum_{\mathbf t\in(\mathbb F_p^\times)^{k-1}}E_{\mathbf t}
    \le
    |A|^k+\frac{|A|^{2k}-|A|^k}{p-1},
\]
and hence, there exists \(\mathbf t\in(\mathbb F_p^\times)^{k-1}\) such that
\[
    E_{\mathbf t}
    \le
    |A|^k+\frac{|A|^{2k}-|A|^k}{p-1}
    =
    \frac{|A|^k(|A|^k+p-2)}{p-1}.
\]
For this choice of \(\mathbf t\), (\ref{CS}) gives the lower bound
\[
    |S_{\mathbf t}|
    \ge
    \frac{|A|^k(p-1)}{|A|^k+p-2},
\]
and hence,
\[
    |\mathbb F_p\setminus S_{\mathbf t}|
    \le
    p-\frac{|A|^k(p-1)}{|A|^k+p-2}
    =
    \frac{p^2-2p+|A|^k}{|A|^k+p-2}
    \le
    \frac{p^2}{|A|^k}+1.
\]
Fix this choice of \(t_2,\ldots,t_k\), and write
$S:=A+t_2A+\cdots+t_kA$ and $H:=\mathbb F_p\setminus S$.
From the choice of $k$,
\begin{equation}
    |H|\le \frac{p^2}{|A|^{n-1}}+1.
    \label{eq:dilated-n-fold-complement-bound}
\end{equation}

We now prove that \(S+sA=\mathbb F_p\) for some
\(s\in\mathbb F_p^\times\). If \(|H|<|A|\), then for any
fixed \(s\in\mathbb F_p^\times\) and any \(x\in\mathbb F_p\), the translate
\(x-sA\) has cardinality \(|A|\), and therefore cannot be contained in \(H\).
Thus, \((x-sA)\cap S\ne\varnothing\) and hence, \(x\in S+sA\). As a result,
\(S+sA=\mathbb F_p\).

We may therefore assume that
$|H|\ge |A|$. Suppose, for contradiction, that
\(S+sA\ne\mathbb F_p\) for every
\(s\in\mathbb F_p^\times\). For each \(s\in\mathbb F_p^\times\), choose
\(x_s\in\mathbb F_p\setminus(S+sA)\). Then
\begin{align}\label{incidence}
    x_s-sA\subseteq H.
\end{align}
For each \(s\in\mathbb F_p^\times\), define the line
\[
    \ell_s:=\{(u,v)\in\mathbb F_p^2:v=x_s-su\}.
\]
From (\ref{incidence}), we get that the point \((a,x_s-sa)\) lies in \(A\times H\) and on
\(\ell_s\), for every \(a\in A\). Thus, each \(\ell_s\) contains at least \(|A|\) points of
\(A\times H\). Let
\[
    \mathcal L:=\{\ell_s:s\in\mathbb F_p^\times\}.
\]
The lines in \(\mathcal L\) have distinct slopes, and thus
\(|\mathcal L|=p-1\), and we get the incidence lower bound
\begin{equation}
    \mathcal I(A\times H,\mathcal L)\ge |A|(p-1).
    \label{eq:dilated-n-fold-incidence-lower}
\end{equation}

We apply Proposition~\ref{prop:sdz-cartesian-product-incidence} with
\(X=A\) and \(Y=H\). The condition \(|X|\le |Y|\) follows from the assumption
$|H|\ge |A|$, and
$|X||\mathcal L|=|A|(p-1)\le p^2$.
It remains to check \(|X||Y|^2\le |\mathcal L|^3\).

By \eqref{eq:dilated-n-fold-complement-bound} and the hypothesis $  |A|\ge C p^{\frac{3}{2n-1}}$,
we have
\[
\begin{aligned}
    \frac{|H|}{(|A|p)^{1/2}}
    \le
    \frac{p^{3/2}}{|A|^{n-\frac12}}
    +\frac{1}{(|A|p)^{1/2}}\le
    C^{-(n-\frac12)}
    +\frac{1}{(|A|p)^{1/2}}.
\end{aligned}
\]
Thus, for any sufficiently small fixed \(\delta>0\), we may first choose
\(C=C(n)\) and then \(p_0=p_0(n)\) sufficiently large so that
\begin{equation}
    |H|\le \delta (|A|p)^{1/2}.
    \label{eq:dilated-n-fold-h-small}
\end{equation}
In particular, since \(|A|\le p\), we have
\[
    |X||Y|^2
    =|A||H|^2
    \le \delta^2|A|^2p
    \le \delta^2p^3
    \le (p-1)^3
    =|\mathcal L|^3
\]
for all sufficiently large \(p\).

Therefore, Proposition~\ref{prop:sdz-cartesian-product-incidence}
applies, and combining it with
\eqref{eq:dilated-n-fold-incidence-lower} gives
\[
    |A|(p-1)
    \le \mathcal I(A\times H,\mathcal L)
   \ll
    \bigl(
    |A|^{3/4}|H|^{1/2}(p-1)^{3/4}
    +(p-1)
    \bigr).
\]
Dividing both sides by \(|A|(p-1)\) and combining \eqref{eq:dilated-n-fold-h-small}, there is a constant $C'>0$ such that
\[
    1
    \leq C'
    \left(
    \delta^{1/2}
    \left(\frac{p}{p-1}\right)^{1/4}
    +\frac1{|A|}
    \right).
\]
Since \(|A|\ge C p^{3/(2n-1)}\), the second term tends to zero as
\(p\to\infty\). Choosing \(\delta\) sufficiently small in terms of
$C'$, the right-hand side is strictly smaller than \(1\)
for all sufficiently large \(p\). This is a contradiction, and the proof
is complete.
\end{proof}

\section{Product-type projections over finite fields} \label{section:Product-type projections}
We now pass from linear images to shifted product images.  The underlying
question is still whether a suitable image of \(A^n\) covers all of
\(\mathbb F_p\), but the algebra is now multiplicative rather than additive.
Thus, the results below should be read in parallel with the two linear projection
theorems from the previous section: the density thresholds are the same, while
the proof must also account for zero factors and for the use of inverses in
\(\mathbb F_p^\times\).

We begin with the first result, the two-fold product analogue of the near-half-density
linear theorem.  

\begin{proof}[Proof of Theorem~\ref{thm:half-density-shifted-product-covering}]
If \(|A|=p\), then there is nothing to prove.  If \(\frac{p}{2}<|A|<p\), choose \(a_0\in A\), set \(s:=-a_0\), and choose \(t\notin -A\).  Define
\(X:=A+s\) and \(Y:=A+t\).  By construction, \(0\in X\), \(0\notin Y\), and
\[
    |X\cap\mathbb F_p^\times|=|A|-1
    \quad\text{and}\quad
    |Y|=|A|.
\]
For any \(\lambda\in\mathbb F_p^\times\), the set
\(\lambda Y^{-1}\) has cardinality \(|A|\), and
\[
    |X\cap\mathbb F_p^\times|+|\lambda Y^{-1}|
    =
    2|A|-1
    >
    p-1,
\]
which implies that
\(\lambda\in XY\), by Proposition ~\ref{prop:CD-ineq}.
Hence, \(\mathbb F_p^\times\subseteq XY\), and since \(0\in X\), we also have \(0\in XY\).  Therefore, \(XY=\mathbb F_p\).

We may therefore assume
\[
    \left(\frac12-\varepsilon_0\right)p<|A|\le \frac p2.
\]
Choose \(a_0\in A\), set \(X:=A-a_0\), \(C:=X\setminus\{0\}\), and
\(B:=\mathbb F_p\setminus A\). 
Suppose, for contradiction, that \(X(A+u)\ne\mathbb F_p\) for every
\(u\in\mathbb F_p\).  Since \(0\in X\), every product set \(X(A+u)\) contains
\(0\).  Hence, for each \(u\), we may choose
\(x_u\in\mathbb F_p^\times\setminus X(A+u)\).  Then, for every \(c\in C\),
\[
    x_u c^{-1}-u\in B.
\]
Fix \(u=0\), and define \(D:=x_0C^{-1}\).  Then, \(D\subseteq B\) and
\(|D|=|A|-1\).  For each \(u\in\mathbb F_p\), let
\[
    \alpha_u:=\frac{x_u}{x_0}\in\mathbb F_p^\times
    \quad\text{and}\quad
    g_u(x):=\alpha_u x-u.
\]
Then, \(g_u(D)\subseteq B\) for every \(u\), and \(g_0=\mathrm{id}_{\mathbb F_p}\).  Since
both \(D\) and \(g_u(D)\) are subsets of \(B\), each of cardinality
\(|A|-1\), we then have
\[
    |D\cap g_u(D)|
    \ge
    2(|A|-1)-(p-|A|)
    =
    3|A|-p-2,
\]
and hence,
\begin{align}\label{esti22}
    |D\triangle g_u(D)|
    =
    2|D|-2|D\cap g_u(D)|
    \le
    p\delta,\quad\text{for all}\,u\in\mathbb F_p,
\end{align}
where $\delta:=(2p-4|A|+2)/p$.
The size assumption on $A$ gives
\begin{equation}
    \delta<4\varepsilon_0+\frac{2}{p}.
    \label{eq:half-density-shifted-product-delta-small}
\end{equation}
For \(u,v\in\mathbb F_p\), set
\(w:=u+\alpha_uv\).  Then,
\[
    g_ug_v(x)=\alpha_u\alpha_vx-u-\alpha_uv
    =
    \alpha_u\alpha_vx-w,
\]
and hence, the map
$h_{u,v}:=g_w^{-1}g_ug_v$
is a dilation.  By the triangle inequality and (\ref{esti22}),
\begin{align}\label{eq:half-density-shifted-product-defect-bound}
    |D\triangle h_{u,v}(D)|&=
    |g_w(D)\triangle g_ug_v(D)|\notag\\
    &\leq
    |g_w(D)\triangle D|
    +|D\triangle g_u(D)|+|g_u(D)\triangle g_ug_v(D)|\notag\\&\le 3p\delta.
\end{align}
We now derive a contradiction by considering two cases.

Suppose \(h_{u,v}(x)=\rho x\) for some \(u,v\in\mathbb F_p\), with
\(\rho\ne1\).  By \eqref{eq:half-density-shifted-product-defect-bound}, we obtain that
$|D\triangle \rho D|\le 3p\delta.$
Observe that, for any \(r\in\mathbb F_p\),
\[
    g_r\circ(\rho)\circ g_r^{-1}(x)
    =
    \rho x+(\rho-1)r.
\]
Using the triangle inequality together with
\eqref{eq:half-density-shifted-product-defect-bound}, we get
\[
\begin{aligned}
    |D\triangle g_r(\rho g_r^{-1}(D))|
    &\le
    |D\triangle g_r(D)|
    +|g_r(D)\triangle g_r(\rho D)| 
    +|g_r(\rho D)\triangle g_r(\rho g_r^{-1}(D))|  \\
    &\le
    5p\delta.
\end{aligned}
\]
Since \(g_r(\rho g_r^{-1}(D))=\rho D+(\rho-1)r\), and 
\((\rho-1)r\) ranges over all of \(\mathbb F_p\), we have
\begin{equation}
    |D\triangle(\rho D+b)|\le 5p\delta,
    \quad
    \text{for all}\,b\in\mathbb F_p.
    \label{eq:half-density-shifted-product-uniform-translate-bound}
\end{equation}
Averaging over \(b\), we have
\[
\begin{aligned}
    2|D|-\frac{2|D|^2}{p}
    =
    \frac1p\sum_{b\in\mathbb F_p}|D\triangle(\rho D+b)| \leq 5 p \delta.
\end{aligned}
\]
Recall that \(|D|=|A|-1\). The size assumption on $A$, together with \eqref{eq:half-density-shifted-product-uniform-translate-bound} and \eqref{eq:half-density-shifted-product-delta-small}, implies that 
\[
\left(\frac{1}{2}-3\varepsilon_0 \right)p \leq
    2|D|-\frac{2|D|^2}{p}
    \leq 5p\delta
    <
    20\varepsilon_0p+10,
\]
for all sufficiently large
\(p\).
For \(\varepsilon_0=10^{-3}\), these two estimates are incompatible for all
sufficiently large \(p\), a contradiction.  

Thus, we may assume all the maps \(h_{u,v}\) are trivial, that is
\[
h_{u,v}=\mathrm{id}_{\mathbb F_p}, \quad
    \text{for all }u,v\in\mathbb F_p,
\]
or equivalently
\(
    g_ug_v=g_{u+\alpha_uv}
\)
for all \(u,v\in\mathbb F_p\).
In this case,
\[
    \Gamma:=\{g_u:u\in\mathbb F_p\}
\]
is closed under composition.  Since the translation term of \(g_u\) is \(-u\),
the maps \(g_u\) are pairwise distinct and thus \(|\Gamma|=p\).  Also,
\(g_0\) is the identity map, so \(\Gamma\) is a subgroup of the affine group of
\(\mathbb F_p\) of order \(p\).
We claim that every element of \(\Gamma\) is a translation.  Indeed, an affine map
\(x\mapsto \alpha x+\beta\) with \(\alpha\ne1\) has a fixed point and is
conjugate to the dilation \(x\mapsto \alpha x\), so its order divides \(p-1\).
But every element of \(\Gamma\) has order dividing \(p\).  Hence,
\(\alpha=1\) for every element of \(\Gamma\). 
Thus, \eqref{esti22} becomes
\[
    |D\triangle(D-u)|\le p\delta,\quad\text{for all}\,
   u\in\mathbb F_p.
\]
Averaging over \(u\), together with \eqref{eq:half-density-shifted-product-delta-small}, we get
\[
    2|D|\left(1-\frac {|D|}{p}\right)
    =
    \frac1p\sum_{u\in\mathbb F_p}|D\triangle(D-u)|
    \le
    p\delta
     <4\varepsilon_0 p+2.
\]
However, the left-hand side is at least \((1/2-3\varepsilon_0)p\) for all sufficiently
large \(p\).  For \(\varepsilon_0=10^{-3}\), these estimates
are incompatible for all sufficiently large \(p\).  This is again a contradiction.

As a consequence, there exists \(u\in\mathbb F_p\) such that
\(X(A+u)=\mathbb F_p\).  Equivalently, with \(s=-a_0\) and \(t=u\), we have
\[
    (A+s)(A+t)=\mathbb F_p.
\]
This completes the proof.
\end{proof}

We next turn to the \(n\)-fold product-type theorem.  Compared with the
\(n\)-fold linear result, the averaging step is now carried out for products of
shifted copies of \(A\), and the complement is measured inside
\(\mathbb F_p^\times\).  The role of zero is separated off at the beginning by
choosing the first shift so that \(0\) belongs to the partial product.  Once the
nonzero complement is made small, the final covering step again reduces to a
Cartesian-product incidence estimate.

\begin{proof}[Proof of Theorem~\ref{thm:shifted-product-covering}]
The trivial cases being clear, assume \(2\le |A|<p\). Choose \(a_0\in A\), set
\(t_1:=-a_0\), and put \(k:=n-1\geq2\). Then $t_1 + A$ contains \(0\).
For \(\mathbf t=(t_2,\ldots,t_k)\), define
\[
    S_{\mathbf t}:=(t_1 + A)(t_2 + A)\cdots(t_k + A),
    \quad
    S_{\mathbf t}^{\ast}:=S_{\mathbf t}\cap\mathbb F_p^\times
    \quad\text{and}\quad
    H_{\mathbf t}:=\mathbb F_p^\times\setminus S_{\mathbf t}^{\ast}.
\]
We will choose \(\mathbf t\) so that the size of \(H_{\mathbf t}\) is small. This is done
by the following two averaging arguments.

\begin{itemize}
    \item \textbf{First averaging argument:}
Let \(\mathbf t=(t_2,\ldots,t_k)\in\mathbb F_p^{k-1}\). For
\(x\in\mathbb F_p^\times\), define
\begin{align}\label{mulexpression}
r_{\mathbf t}(x)
    :=
    \#\left\{(a_1,\ldots,a_k)\in A^k:
    \prod_{i=1}^k(a_i+t_i)=x\right\},
\end{align}
and set the multiplicative energy
\(E_{\mathbf t}:=\sum_{x\in\mathbb F_p^\times}r_{\mathbf t}(x)^2\). Note that
\begin{align}\label{L^1}
\sum_{x\in\mathbb F_p^\times} r_{\mathbf t}(x)\ge (|A|-1)^k;
\end{align}
the lower bound of the $L^1$-norm can be deduced from the choice of $t_1$ and the fact that for each \(i\ge2\), at most one value of
\(a_i\in A\) satisfies \(a_i+t_i=0\).

Averaging \(E_{\mathbf t}\) over \(\mathbb F_p^{k-1}\), we have
\[
\sum_{\mathbf t\in\mathbb F_p^{k-1}}E_{\mathbf t}
=
\sum_{\mathbf a,\mathbf b\in A^k}
\#\left\{\mathbf t\in\mathbb F_p^{k-1}:
\prod_{i=1}^k(a_i+t_i)
=
\prod_{i=1}^k(b_i+t_i)\ne0
\right\},
\]
where \(t_1=-a_0\) is fixed. The diagonal pairs
\(\mathbf a=\mathbf b\) contribute at most
\[
    p^{k-1}|A|^k,
\]
since there are \(|A|^k\) such pairs and at most \(p^{k-1}\)
choices of \(\mathbf t\) for each pair.

For a non-diagonal pair \(\mathbf a\ne\mathbf b\), we count
the choices of \(\mathbf t\) for which
\[
    \prod_{i=1}^k(a_i+t_i)
    =
    \prod_{i=1}^k(b_i+t_i)
    \ne0.
\]
If \(a_2=b_2,\ldots,a_k=b_k\), then the nonzero product
condition forces \(a_1=b_1\), contrary to non-diagonality.
Hence, \(a_j\ne b_j\) for some \(j\in\{2,\ldots,k\}\).
Fix all \(t_i\) with \(i\ne j\), and write
\[
    P_a:=\prod_{i\ne j}(a_i+t_i)\quad\text{and}\quad
    P_b:=\prod_{i\ne j}(b_i+t_i).
\]
The above equality becomes
\[
    P_a(a_j+t_j)=P_b(b_j+t_j).
\]
Since the products are nonzero, both \(P_a\) and \(P_b\)
are nonzero. Moreover, \(P_a\ne P_b\), since otherwise
the equality would imply \(a_j=b_j\).
Thus,
\[
    t_j=\frac{P_b b_j-P_a a_j}{P_a-P_b}.
\]
Therefore, each non-diagonal pair contributes at most
\(p^{k-2}\) choices of \(\mathbf t\).

Combining the
diagonal and non-diagonal contributions, we obtain
\[
    \sum_{\mathbf t\in\mathbb F_p^{k-1}}E_{\mathbf t}
    \le
    p^{k-1}|A|^k
    +p^{k-2}(|A|^{2k}-|A|^k).
\]
Dividing by \(p^{k-1}\), there exists
\(\mathbf t\in\mathbb F_p^{k-1}\) such that
\[
    E_{\mathbf t}
    \le
    |A|^k+\frac{|A|^{2k}-|A|^k}{p}.
\]
For this choice of \(\mathbf t\), the Cauchy--Schwarz inequality together with \eqref{L^1} gives
\[
    |S_{\mathbf t}^{\ast}|
    \ge
    \frac{(|A|-1)^{2k}}
    {|A|^k+\frac{|A|^{2k}-|A|^k}{p}},
\]
which implies that, for all sufficiently large $p$,
\begin{align}\label{compare1}
    |\mathbb F_p^\times\setminus S_{\mathbf t}^{\ast}|
    \ll_n
    \frac{p^2}{|A|^k}+\frac{p}{|A|}+1.
\end{align}

\item \textbf{Second averaging argument:}
Now restrict \(t_2,\ldots,t_k\) to $U:=\mathbb F_p\setminus(-A)$, with
$|U|=p-|A|$. For \(\mathbf t\in U^{k-1}\), all factors \(a_i+t_i\), \(i\ge2\), are
nonzero for every \(a_i\in A\). The only possible zero factor in (\ref{mulexpression}) is
\(a_1+t_1=a_1-a_0\). Hence,
\begin{align}\label{L^1again}
\sum_{x\in\mathbb F_p^\times} r_{\mathbf t}(x)
    =
    (|A|-1)|A|^{k-1}
    =:T,
    \quad
    \text{for every }\mathbf t\in U^{k-1}.
\end{align}
Averaging \(E_{\mathbf t}\) over \(U^{k-1}\), the diagonal contribution is
\(|U|^{k-1}T\), and the same linearity argument as above shows that each
non-diagonal pair contributes at most \(|U|^{k-2}\) choices of \(\mathbf t\).
Thus, there exists \(\mathbf t\in U^{k-1}\) such that
\[
    E_{\mathbf t}
    \le
    T+\frac{T^2-T}{p-|A|}.
\]
By the definition of multiplicative energy, the Cauchy--Schwarz inequality, and \eqref{L^1again}, we obtain
\[
    |S_{\mathbf t}^{\ast}|
    \ge
    \frac{T^2}{T+\frac{T^2-T}{p-|A|}}
    =
    \frac{T(p-|A|)}{T+p-|A|-1}.
\]
Since \(T=(|A|-1)|A|^{k-1}\asymp |A|^k\), this gives
\begin{align}\label{compare2}
|\mathbb F_p^\times\setminus S_{\mathbf t}^{\ast}|
    \ll_n
    |A|+\frac{p^2}{|A|^k}+1.
\end{align}

\end{itemize}
Taking the better of the two choices obtained from \eqref{compare1} and \eqref{compare2}, we may choose \(t_2,\ldots,t_k\) such
that, with
\[
    0\in S:=(t_1 + A)(t_2 + A)\cdots(t_k + A)
    \quad\text{and}\quad
    S^\ast:=S\cap\mathbb F_p^\times,
\]
we have
\begin{equation}
  |\mathbb F_p^\times\setminus S^\ast|
    \ll_n
    \frac{p^2}{|A|^k}
    +
    \min\left\{|A|,\frac{p}{|A|}\right\}
    +1.
    \label{eq:shifted-product-final-complement-bound}
\end{equation}

We next apply the incidence estimate, as in the linear projection case. Define
\(
    H:=\mathbb F_p^\times\setminus S^\ast.
\)
We now show that \(S(t_n + A)=\mathbb F_p\) for some \(t_n\in\mathbb F_p\).
If \(|H|<|A|\), choose \(t_n\in\mathbb F_p\setminus(-A)\). Then
\(t_n + A\subseteq\mathbb F_p^\times\) and for every
\(x\in\mathbb F_p^\times\), the set \(x(t_n + A)^{-1}\) has cardinality \(|A|\),
and hence, cannot be contained in \(H\). Thus, \(x\in S(t_n + A)\), that is
$\mathbb F_p^\times\subseteq S(t_n + A)$. Since
\(0\in S\), we also have \(0\in S(t_n + A)\), and so \(S(t_n + A)=\mathbb F_p\).

We may therefore assume that $|H|\ge |A|$.
Suppose, for contradiction, that \(S(t + A)\ne\mathbb F_p\) for every
\(t\in\mathbb F_p\). Since \(0\in S\), every missing element is nonzero.
For each \(t\in\mathbb F_p\), choose
\(x_t\in\mathbb F_p^\times\setminus S(t + A)\). Then, for every
\(a\in A\) with \(t+a\ne0\), we have
\begin{align}\label{prepincidence}
   x_t^{-1}(t+a)\in H^{-1}.
\end{align}
For each \(t\in\mathbb F_p\), define the line
\[
    \ell_t:=\{(u,v)\in\mathbb F_p^2:\,v=x_t^{-1}u+tx_t^{-1}\}.
\]
Then, from (\ref{prepincidence}), for every \(a\in A\) with \(t+a\ne0\), the point
$\left(a,x_t^{-1}(a+t)\right)$
lies in \(A\times H^{-1}\) and on \(\ell_t\). Thus, each \(\ell_t\) contains at
least \(|A|-1\) points of \(A\times H^{-1}\). The lines \(\ell_t\) are distinct. Let
\[
    \mathcal L:=\{\ell_t:t\in\mathbb F_p\}.
\]
Then \(|\mathcal L|=p\), and we have the incidence lower bound
\begin{equation}
    \mathcal I(A\times H^{-1},\mathcal L)\ge (|A|-1)p.
    \label{eq:shifted-product-incidence-lower}
\end{equation}

We apply Proposition~\ref{prop:sdz-cartesian-product-incidence} with
\(X=A\) and \(Y=H^{-1}\). 
Since \(H\subseteq\mathbb F_p^\times\), the inversion map
\(x\mapsto x^{-1}\) is a bijection from \(H\) onto \(H^{-1}\).
Thus, \(|H^{-1}|=|H|\).
 The condition \(|X|\le |Y|\)
follows from the assumption, $|H|\ge |A|$, and
\(|X||\mathcal L|=|A|p\le p^2\).

It remains to check \(|X||Y|^2\le|\mathcal L|^3\), that is,
\(|A||H|^2\le p^3\). From \eqref{eq:shifted-product-final-complement-bound}
and \(k=n-1\),
\[
    |H|
    \ll_n
    \frac{p^2}{|A|^{n-1}}
    +
    \min\left\{|A|,\frac{p}{|A|}\right\}
    +1.
\]
Using the estimate $  \min\left\{|A|,\frac{p}{|A|}\right\}\le p^{1/2}$,
we obtain
\[
    |H|
    \ll_n
    \frac{p^2}{|A|^{n-1}}+p^{1/2}+1,
\]
and hence, by the hypothesis that $|A|\ge C p^{\frac{3}{2n-1}}$,
we have
\[
\begin{aligned}
    \frac{|H|}{(|A|p)^{1/2}}
    \ll_n
    \frac{p^{3/2}}{|A|^{n-\frac12}}
    +\frac{1}{|A|^{1/2}}
    +\frac{1}{(|A|p)^{1/2}}\le
    C^{-(n-\frac12)}
    +\frac{1}{|A|^{1/2}}
    +\frac{1}{(|A|p)^{1/2}}.
\end{aligned}
\]
Since \(|A|\ge C p^{3/(2n-1)}\), the last two terms tend to zero
as \(p\to\infty\). Thus, for any sufficiently small fixed
\(\delta>0\), we may first choose \(C=C(n)\) and then
\(p_0=p_0(n)\) sufficiently large so that
\begin{equation}
    |H|\le \delta (|A|p)^{1/2}.
    \label{eq:shifted-product-h-small}
\end{equation}
In particular, since \(|A|\le p\), we have
\[
    |X||Y|^2
    =|A||H|^2
    \le \delta^2|A|^2p
    \le \delta^2p^3
    \le p^3
    =|\mathcal L|^3.
\]
Therefore, Proposition~\ref{prop:sdz-cartesian-product-incidence}
applies, and combining it with
\eqref{eq:shifted-product-incidence-lower} gives
\[
    (|A|-1)p
    \le \mathcal I(A\times H^{-1},\mathcal L)
    \ll
    \left(
    |A|^{3/4}|H|^{1/2}p^{3/4}+p
    \right).
\]

Dividing both sides by \(|A|p\) and combining \eqref{eq:shifted-product-h-small}, there is a constant $C'>0$ such that
\[
    1-\frac1{|A|}
    \le C'
    \left(
    \delta^{1/2}+\frac1{|A|}
    \right).
\]
Since \(|A|\to\infty\) as \(p\to\infty\), choosing
\(\delta\) sufficiently small in terms of \(C'\),
the right-hand side is strictly smaller than the left-hand
side for all sufficiently large \(p\). This is a contradiction,
and the proof is complete.
\end{proof}

\section{Euclidean product-type projections} \label{section:Euclidean product-type projections}

In this section, we prove the Euclidean analogue for product-type
projections. Since it suffices to work with a compact subset
$K\subseteq A$ of sufficiently large Hausdorff dimension, we may
assume that $A$ is compact. The main idea is to pass to logarithmic
coordinates: after choosing the shifts $t_1,\ldots,t_n$ sufficiently
large, the shifted product set
\[
    (t_1+A)(t_2+A)\cdots(t_n+A)
\]
is transformed into the sumset
\[
    \log(t_1+A)+\log(t_2+A)+\cdots+\log(t_n+A).
\]
We then choose the shifts so that the corresponding convolution
measure has a continuous density.

We recall the measure-theoretic material used below; see
Mattila~\cite[Sections~2.5 and~3.5]{MR3617376}.  Let $E\subseteq\mathbb R$ be Borel. For $s>0$, a nonzero finite
Borel measure $\mu$ supported on $E$ is called an $s$-Frostman
measure if there exists $C>0$ such that
\[
    \mu(B(x,r))\le Cr^s
    \qquad\text{for all }x\in\mathbb R,\ r>0.
\]
By Frostman's lemma, for every $0<s<\dim_H E$, there is an
$s$-Frostman measure supported on a compact subset
of $E$. The ball-growth condition implies that such a measure has
finite $\sigma$-energy for every $0<\sigma<s$. In particular, given
$0<\sigma<\dim_H E$, we may choose $s$ with $\sigma<s<\dim_H E$
and obtain a compact set $K\subseteq E$ and a Borel probability
measure $\mu$ supported on $K$ such that
\[
    I_\sigma(\mu):=\iint |x-y|^{-\sigma}\,d\mu(x)\,d\mu(y)<\infty.
\]
We will only use this finite-energy property. Notice also that such a measure has no atoms, since an atom would make the integral defining
\(I_\sigma(\mu)\) diverge along the diagonal.

\begin{proof}[Proof of Theorem~\ref{thm:continuous-shifted-product-interior}]
Choose
\[
    \frac{2}{n}<\sigma<\dim_H A.
\]
Since \(A\subseteq\mathbb R\), we may decrease \(\sigma\) if necessary and
assume \(0<\sigma<1\).  By the energy formulation recalled above, there are a
compact set \(K\subseteq A\) and a Borel probability measure \(\mu\) supported
on \(K\) such that $I_\sigma(\mu)<\infty$. In particular, \(\mu\) has no atoms.

Since $K\subseteq A$ is compact, it is bounded, and it
suffices to prove the conclusion for $K$.
We may therefore choose $R>0$ sufficiently large so that
$t+x>0$ for all $t\in[R,2R]$ and $x\in K$.
  For \(t\in[R,2R]\), let
\[
    \phi_t(x):=\log(t+x)
    \quad\text{and}\quad
    \nu_t:=(\phi_t)_\#\mu.
\]
Then, for all $t\in[R,2R]$, \(\nu_t\) is a probability measure supported on \(\log(t+K)\).

We first prove the averaged Fourier decay estimate:
\begin{equation}
    \int_R^{2R} |\widehat{\nu_t}(\xi)|^2\,dt
    \ll_{\mu,R,\sigma}
    (1+|\xi|)^{-\sigma},
    \quad
    \text{for all}\,\xi\in\mathbb R.
    \label{eq:euclidean-product-avg-fourier-decay}
\end{equation}
For \(|\xi|\le1\), \eqref{eq:euclidean-product-avg-fourier-decay} follows directly from \(|\widehat{\nu_t}(\xi)|\le1\).  It suffices to verify \eqref{eq:euclidean-product-avg-fourier-decay} under the assumption that
\(|\xi|\ge1\).  
By the definition of the pushforward measure $\nu_t$, we have
\[
    \widehat{\nu_t}(\xi)
    =
    \int_K e^{-2\pi i\xi\log(t+x)}\,d\mu(x).
\]
Thus, by Fubini's theorem,
\begin{align}
    \int_R^{2R}|\widehat{\nu_t}(\xi)|^2\,dt
    &=
    \iint\int_R^{2R}
    e^{-2\pi i\xi(\log(t+x)-\log(t+y))}
    \,dt\,d\mu(x)\,d\mu(y).
    \label{eq:fourier-average}
\end{align}

For \(x\ne y\) in $K$ and \(t\in[R,2R]\), write
\[
    \Psi_{x,y}(t):=\log(t+x)-\log(t+y)
\]
for the phase function.
Then, for the distinct $x,y\in K$ and \(t\in[R,2R]\),
\[
    \Psi'_{x,y}(t)
    =
    \frac{y-x}{(t+x)(t+y)}\quad\text{and}\quad |\Psi'_{x,y}(t)|\asymp\frac{|x-y|}{R^2}\asymp_R |x-y|.
\]
After increasing \(R\) if necessary, \(\Psi'_{x,y}\) is monotone in \(t\).
Hence, the first-derivative form of van der Corput's lemma \cite[Chapter~VIII]{MR1232192} gives, for $\mu$-almost every $x,y$,
\begin{align}\label{staionary}
\left|
    \int_R^{2R}
    e^{-2\pi i\xi\Psi_{x,y}(t)}
    \,dt
    \right|
    \ll_R
    \min\left\{1,\frac{1}{|\xi||x-y|}\right\}.
\end{align}

Since \(0<\sigma<1\), combining \eqref{staionary} with \eqref{eq:fourier-average} gives

\[
\begin{aligned}
    \int_R^{2R}|\widehat{\nu_t}(\xi)|^2\,dt
    &\ll_R
    |\xi|^{-\sigma}
    \iint |x-y|^{-\sigma}\,d\mu(x)d\mu(y)      =
    |\xi|^{-\sigma} I_\sigma(\mu),\quad\text{for all} |\xi|\geq1,
\end{aligned}
\]
which proves \eqref{eq:euclidean-product-avg-fourier-decay}, as desired.

Choose \(\gamma>1/2\) such that
$2\gamma+1<n\sigma.$  By Tonelli's theorem and
\eqref{eq:euclidean-product-avg-fourier-decay}, we have
\[
\begin{aligned}
\int_{[R,2R]^n}
\int_{\mathbb R}
(1+|\xi|)^{2\gamma}
\prod_{j=1}^n |\widehat{\nu_{t_j}}(\xi)|^2
\,d\xi\,dt_1\cdots dt_n                            
\ll_{\mu,R,\sigma}
\int_{\mathbb R}
(1+|\xi|)^{2\gamma-n\sigma}
\,d\xi
<\infty,
\end{aligned}
\]
where the finiteness is guaranteed by the choice of $\gamma.$  Hence, there exists
\((t_1,\ldots,t_n)\in[R,2R]^n\) such that
\begin{equation}
    \int_{\mathbb R}
    (1+|\xi|)^{2\gamma}
    \prod_{j=1}^n |\widehat{\nu_{t_j}}(\xi)|^2
    \,d\xi
    <\infty;
    \label{eq:euclidean-product-sobolev-fourier-bound}
\end{equation}
In other words, if we define
\[
    \rho:=\nu_{t_1}*\cdots *\nu_{t_n},
\]
then \eqref{eq:euclidean-product-sobolev-fourier-bound} implies that \(\rho\) has a density in \(H^\gamma(\mathbb R)\).  Since \(\gamma>1/2\), Sobolev
embedding gives a continuous representative of this density.

Since \(\rho\) has total mass \(1\), this continuous density is not
identically zero.  As \(\rho\) is a positive measure, the density is
nonnegative, and hence, it is positive on some nonempty open interval
\(J\subseteq\mathbb R\).  Thus, \(J\subseteq\operatorname{supp}(\rho)\).  On the
other hand,
\[
    \operatorname{supp}(\rho)
    \subseteq
    \log(t_1 + K)+\cdots+\log(t_n + K).
\]
Therefore,
\[
    J\subseteq \log(t_1 + K)+\cdots+\log(t_n + K),
\]
which implies that
\[
    \exp(J)
    \subseteq
    (t_1 + K)(t_2 + K)\cdots(t_n + K).
\]
Since \(K\subseteq A\), the shifted product set
\[
    (t_1 + A)(t_2 + A)\cdots(t_n + A)
\]
contains the nonempty open interval \(\exp(J)\).  This proves the theorem.
\end{proof}
\medskip
\noindent{\bf AI Disclosure:}
The authors used GPT-5.5 for language editing and exploratory discussion. In particular, the key idea in Theorem~\ref{thm:half-density-dilated-sumset-covering} was discovered through a discussion with GPT. The authors take full responsibility for the mathematical content of the paper.

\medskip

\noindent {\bf Acknowledgment:}
The authors would like to thank the anonymous referee for many valuable comments and suggestions, which helped improve the presentation and strengthen the main results of this paper. We are also grateful to Ilya Shkredov for his valuable comments,
for sharing his counterexample to
Question~\ref{ques:two-dimensional-exponent-saving}, and for allowing
us to present it in Remark~\ref{rem:shkredov-example}.

The first author is supported by the MOE Taiwan-Caltech Fellowship during the conduct of this research. The second and third authors are supported by the National Science and Technology Council (NSTC) under Grant No.~111-2115-M-002-010-MY5.

\end{document}